\documentclass{article}

\usepackage{graphicx}%
\usepackage{multirow}%
\usepackage{amsmath,amssymb,amsfonts}%
\usepackage{amsthm}%
\usepackage{mathrsfs}%
\usepackage[title]{appendix}%
\usepackage{xcolor}%
\usepackage{textcomp}%
\usepackage{manyfoot}%
\usepackage{booktabs}%
\usepackage{algorithm}%
\usepackage{algorithmicx}%
\usepackage{algpseudocode}%
\usepackage{listings}

\newtheorem{theorem}{Theorem}[section]
\newtheorem{lemma}[theorem]{Lemma}
\newtheorem{proposition}[theorem]{Proposition}
\newtheorem{corollary}[theorem]{Corollary}
\newtheorem{exam}[theorem]{Example}

\newtheorem{ques}{Question}

\newcommand{\K}{\mathcal{K}}

\newcommand{\tq}{\ge_T}

\DeclareMathOperator{\cof}{cof}
\newcommand{\cl}[1]{\overline{#1}}
\newcommand{\down}[1]{\, \downarrow \! \! #1\,}

\begin{document}

\title{Directed Sets of Topology -- Tukey Representation and Rejection}

\author{Ziqing Feng and Paul Gartside}

\date{August 2023}

\maketitle

\abstract{
Every directed set is Tukey equivalent to   (a) the family of all compact subsets, ordered by inclusion, of a (locally compact) space, to
(b) a neighborhood filter, ordered by reverse inclusion, of a point (of a compact space, and of a topological group),
and to (c) the universal uniformity, ordered by reverse inclusion, of a space. 
Two directed sets are Tukey equivalent if they are cofinally equivalent in the sense that they can both be order embedded \emph{cofinally} in a third directed set.

In contrast, any totally bounded uniformity is Tukey equivalent to $[\kappa]^{<\omega}$, the collection of all finite subsets of $\kappa$, where $\kappa$ is the cofinality of the uniformity. All other Tukey types are `rejected' by totally bounded uniformities. 
Equivalently, a compact space $X$ has weight (minimal size of a base) equal to $\kappa$ if and only if the neighborhood filter of the diagonal is Tukey equivalent to $[\kappa]^{<\omega}$.

A number  of  questions from the literature are answered with the aid of the above results.


}

\section{Introduction}
Two directed sets are \emph{Tukey equivalent} if they are cofinally equivalent in the sense that they can both be order embedded \emph{cofinally} in a third directed set.
Tukey equivalence  \cite{Tukey} was originally introduced, early in the 20th century, as a tool to understand convergence in general topological spaces, however it was quickly seen to have
broad applicability in comparing partial orders. 
 Key to the utility of  Tukey equivalence  (and more generally, Tukey order) is that because it focuses on what happens 
cofinally  it is sufficiently coarse to allow comparison of
very different directed sets, but nevertheless preserves many order invariants. 
Fremlin \cite{Fr3} was the first to realize the
relevance of the Tukey order in analysis, showing that a fundamental result of Bartoszynski 
and Raisonnier \& Stern on additivity of the measure and category ideals
was due to the Tukey order relation between the relevant ideals.

Directed sets arise naturally in topology  and in a variety of contexts. (In this paper all topological spaces are Tychonoff.)
Here we show that every directed set can be represented, up to Tukey equivalence, by such a topological directed set. 
Specifically, Theorem~\ref{represent} states that 
for every directed set $P$ there is a locally compact space $X_P$, a compact space $K_P$ and point $x_P$, and a space $Y_P$ such that $P$ is Tukey equivalent to $\K(X)$ (all compact subsets of $X$ ordered by inclusion), $\mathcal{N}_{x_P}^{K_P}$ (all open neighborhoods of $x_P$ in $K_P$ ordered by reverse inclusion), and $\mathcal{U}_{Y_P}$ (the universal uniformity of $Y_P$ ordered by reverse inclusion).
In principle, then, the study of arbitrary directed sets up to Tukey equivalence can be restricted simply to that of $\K(X)$ (or neighborhood filters of points in a compact space, or universal uniformities). 
Fortunately the Tukey types of the directed set $\K(X)$ have been studied  intensively  \cite{Fr2, GM,  GartMama,GM2, EFHST}.

Strikingly, in the opposite direction we show, see Theorem~\ref{th:reject1}, that any totally bounded uniformity is Tukey equivalent to $[\kappa]^{<\omega}$, the collection of all finite subsets of $\kappa$, the cofinality of the uniformity - all other Tukey types are `rejected'. A more topological interpretation of this result, given in Theorem~\ref{th:reject3}, is that  a compact space $X$ has weight (minimal size of a base) equal to $\kappa$ if and only if the neighborhood filter of the diagonal is Tukey equivalent to $[\kappa]^{<\omega}$. 
Some variations and open problems on this theme are discussed in Section~\ref{ssec:VP}.

A space $X$ is said to have a \emph{$P$-base}, where $P$ is a directed set, if  every point $x$ has a neighborhood base, $\{U_p : p \in P\}$, where $U_p \subseteq U_{p'}$ if $p \ge p'$. 
It is easy to see that a space with a compatible uniformity Tukey equivalent to some $P$ has a $P$-base, and it is immediate that if a space has a $P$-base then it contains a point whose neighborhood filter is Tukey equivalent to $P$. 
It follows from the results outlined above that although in general (even in topological groups) there are spaces with a $P$-base for every $P$ (up to Tukey equivalence), for \emph{compact} spaces having a $P$-base for some $P$ is delicately balanced between `everything goes' and `everything rejected except $[\kappa]^{<\omega}$'. 
This helps explain the recent surge of interest \cite{Banakh2019, BL18, DowFeng, Feng_Pbase} in spaces with $P$-bases, especially compacta, for certain nice $P$. 
One particular outstanding question is whether compact, scattered spaces, of countable scattered height with a $P$-base where $P$ has calibre $(\omega_1,\omega)$ are necessarily countable. Theorem~\ref{th:hM} gives a positive answer under a mild additional condition.

In the final section we  apply the results above to answer a number (eleven) of  questions from the literature.

\section{Representing Directed Sets}

\subsection{Directed Sets from Topology}\label{ssec:DST}

For any space $X$, $\K(X)$ will denote the directed set of all compact subsets of $X$, ordered by inclusion.  If $A$ is a subset of $X$, then $\mathcal{N}_A^X$ will denote the neighborhood filter of $A$ in $X$, ordered by reverse inclusion. We abbreviate $\mathcal{N}_{\{x\}}^X$, the neighborhood filter of a point $x$ in $X$, to $\mathcal{N}_x^X$. Along with neighborhood filters of points, we pay particular attention to the neighborhood filter, $\mathcal{N}_\Delta^{X^2}$, of the diagonal, and certain subfilters, specifically compatible uniformities. 
Recall that a (compatible) uniformity on $X$ is a subfilter, say $\mathcal{U}$, of $\mathcal{N}_\Delta^{X^2}$ (such that for every $x$ in $X$ the family of sets $U[x]=\{y : (x,y) \in U\}$, where $U$ is in $\mathcal{U}$, is a neighborhood base at $x$ in $X$) where, for any $U$ we have that $U^{-1}=\{(y,x) : (x,y) \in U\}$ is in $\mathcal{U}$ and there is a $V$ in $\mathcal{U}$ such that $V \circ V \subseteq U$, where $V \circ V =\{(x,z) : (x,y), (y,z) \in V\}$.  
The universal uniformity, $\mathcal{U}_X$, is the finest compatible uniformity on $X$. We note that if $X$ is paracompact then the universal uniformity and neighborhood filter of the diagonal coincide.
A uniformity, $\mathcal{U}$, on $X$ is totally bounded if for every $U$ in $\mathcal{U}$ there is a finite subset $F$ of $X$ such that $\{U[x] : x \in F\}$ covers $X$. Every space has at least one compatible totally bounded uniformity. 
Every uniformity, $\mathcal{U}$, has a completion, $\widehat{\mathcal{U}}$, defined on a superset, $\widehat{X}$, of $X$, and an alternative description of totally bounded uniformities are those whose completion is compact. (See \cite{engelking} for a general reference on all topics mentioned above.)

Let $P$ and $Q$ be directed sets. Then $Q$ is a \emph{Tukey quotient} of $P$, denoted $P \tq Q$, if there is a \emph{cofinal map} $\phi$ from  $P$ to $Q$, in other words: for every cofinal subset $C$, say, of $P$ we have that $\phi(C)$ is cofinal in $Q$. Any map $\phi: P \to Q$ which is order-preserving and has cofinal image is a cofinal map. 
A map $\psi: Q \to P$ is a called a \emph{Tukey map} if for every unbounded subset, $U$ say, of $Q$ we have that $\psi(U)$ is unbounded in $P$. It turns out that there is a Tukey quotient from $P$ to $Q$ if and only if there is a Tukey map from $Q$ to $P$. 
Two directed sets $P$ and $Q$ are \emph{Tukey equivalent}, denoted $P =_T Q$, if and only if $P \tq Q$ and $Q \tq P$. (This is equivalent to the definition of Tukey equivalence stated in the introduction, namely $P$ and $Q$ embed cofinally in a third directed set.)

 \begin{lemma}\label{l:complete_Tukey}
 Let $\mathcal{U}$ be a uniformity on $X$. Then the uniformity and its completion are Tukey equivalent, $\mathcal{U} =_T \widehat{\mathcal{U}}$. 
 \end{lemma}
To see this, 
define $\phi : \mathcal{U} \to \widehat{\mathcal{U}}$ by $\phi(U)=\cl{U}$ where the closure is taken in $\widehat{X}$, and $\psi:\widehat{\mathcal{U}} \to \mathcal{U}$ by $\psi(\widehat{U})=\widehat{U} \cap X^2$. 
Then $\phi$ and $\psi$ are order-preserving and have cofinal image.

Note that in a topological group $G$, with identity element $e$, the left translates (say) of neighborhoods of the identity gives a compatible uniformity, $\mathcal{U}_L$, which is order isomorphic to $\mathcal{N}_e^G$.
Also, recall that in a compact space all compatible uniformities are equal, and coincide with the family of all neighborhoods of the diagonal. 
Now we show that \emph{any} directed set can be represented, up to Tukey equivalence, by each of the types of directed sets arising in topology introduced above. 

\begin{theorem}\label{represent}
Let $P$ be a directed set. 

(1) $P =_T \K(X_P)$ for some locally compact Hausdorff space $X_P$.

(2) $P =_T \mathcal{N}_{x_P}^{K_P}$ for some compact Hausdorff space $K_P$ and point $x_P$ in $K_P$.

(3) $P =_T \mathcal{N}_e^{G_P}$ for some topological group $G_P$ with identity $e$.

(4) $P \times \omega =_T \mathcal{N}_0^{L_P}$ for some locally convex topological vector space $L_P$.

(5) $P =_T \mathcal{U}_{Y_P}$, the fine uniformity, for some space $Y_P$.

(6) $P =_T \mathcal{N}_\Delta^{Y_P^2}$, for some space $Y_P$.
\end{theorem}

\begin{proof} First we establish (1). Let $D(P)$ be $P$ with the discrete topology. For a $p$ in $P$ let $K_p =\cl{\down{p}}$ where $\down{p}=\{ p' \in P : p' \le p\}$ and the closure is in $\beta D(P)$. Note that $K_p$ is compact and open. 
Let $X_P = \bigcup \{K_p : p \in P\}$ considered as a subspace of $\beta D(P)$. Then $X_P$ is locally compact and Hausdorff.
The map $\phi(p)= K_p$ is an order-embedding of $P$ into $\K(X_P)$ whose image is cofinal. To see cofinality note that $\{K_p : p \in P\}$ is an open cover of $X_P$ so any compact subset, $K$ say, is contained in finitely many $K_{p_1},\ldots, K_{p_n}$, and now if $p$ is an upper bound of $p_1, \ldots, p_n$ then $K \subseteq \phi(p)$. 
Hence $P =_T \K(X_P)$.

Now for (2). By (1) we know  there is a locally compact Hausdorff space $X_P$ such that $P$ and $\K(X_P)$ are Tukey equivalent. Let $K_P$ be the one-point compactification of $X_P$, and let $x_P$ be the point at infinity. Then clearly $(\K(X_P),\subseteq)$ and $(\mathcal{N}_{x_P}^{K_P},\supseteq)$ are order isomorphic, and so Tukey equivalent.

For (3), let $X_P$ be the space in part (1), so $\K(X_P) =_T P$, and note it is zero dimensional. Set $G_P=C_k(X_P, \mathbb{Z}_2)$, the group under co-ordinatewise addition modulo $2$ of all continuous maps of $X_P$ into the discrete two point space, with the compact-open topology. 
The identity of $G_P$ is $\mathbf{0}$, the constant zero function. Standard basic open neighborhoods of $\mathbf{0}$ have the form $B(K)=\{ f \in C(X,\mathbb{Z}_2) : f(K)=0\}$, for compact subsets $K$ of $X_P$. 
Clearly, if $K \subseteq K'$ then $B(K) \supseteq B(K')$, and if $x \in K \setminus K'$ then the characteristic function of a clopen neighborhood of $x$ disjoint from $K'$ is in $B(K')$ but not $B(K)$. 
Thus $\phi : \K(X_P) \to \mathcal{N}_\mathbf{0}^{G_P}$ is an order isomorphism of $(\K(X_P), \subseteq)$ with a cofinal subset of $(\mathcal{N}_\mathbf{0}^{G_P},\supseteq)$, and thus this latter directed set is Tukey equivalent to $P$.

For (4), let $X_P$ be the space of part (1), and set $L_P=C_k(X_P)$. As shown in \cite{GartMorgan_PN}  $\mathcal{N}_0^{L_P} =_T \K(X_P) \times \omega$, and the claim follows.

Finally, we establish (5) and (6). We can suppose that $P$ is the neighborhood filter, $\mathcal{N}_{x_P}^{X_P}$, of part (2). 
Rename the point $x_P$ to be $y_P$.  
Let $Y_P$ have underlying set $X_P$, and topology obtained from $X_P$ by isolating all points except $y_p$. 
Then $\mathcal{N}_{y_p}^{Y_P} =_T P$.   
Now $Y_P$ is paracompact, so $\mathcal{U}_{Y_P} = \mathcal{N}_{\Delta}^{Y_P^2}$. The collection of all $\Delta \cup U^2$ where $U$ is in $\mathcal{N}_{y_P}^{Y_P}$, is cofinal in $\mathcal{N}_{y_p}^{Y_P}$ and order isomorphic to $P$. Claims (5) and (6) follow. 
\end{proof}

Regarding part~(4) above we know (see \cite{GM2}) which directed sets $Q$ are Tukey equivalent to some $P \times \omega$.
 \begin{lemma} Let $Q$ be a directed set. Then the following are equivalent: (i) $Q=_T P \times \omega$ for some directed set $P$, (ii) $Q \tq \omega$, and (iii) $Q$ is not countably directed.
 \end{lemma}

\subsection{Cofinality, Calibres, and Other Properties of $\K(X)$}\label{ssec:CCK}

The \emph{cofinality} of a directed set $P$, denoted $\cof(P)$, is the minimal size of a cofinal subset of $P$. 
If $P \tq Q$ then $\cof(P) \ge \cof(Q)$. Indeed, $\cof(P) \le \kappa$ if and only if $[\kappa]^{<\omega} \tq P$. 
A directed set $P$ has \emph{calibre $(\mu,\lambda)$} where $\mu$ and $\lambda$ are cardinals, $\mu$ regular, if every $\mu$-sized subset of $P$ contains a $\lambda$-sized subset which is bounded. If $P \tq Q$ and $P$ has calibre $(\mu,\lambda)$ then so does $Q$. 
Indeed, $P \not\tq [\mu]^{<\omega}$ if and only if $P$ has calibre $(\mu,\omega)$.

A \emph{topological directed set} is a directed set with a topology.  For example, if $X$ is any space, then $\K(X)$ is directed as usual by inclusion,  and is naturally equipped with the Vietoris topology.  A topological directed set $P$ is said to be \emph{KSB} (compact sets bounded) if every compact subset of $P$ is bounded (above), and \emph{DK}  (down sets compact) if every down set, $\down{p} = \{p' \in P : p' \le p\}$, of $P$ is compact. 
Observe that $\K(X)$ is KSB and DK.

\begin{lemma} \label{l:ce_to_om1om}
Let $Q$ be a topological directed set. 
If $Q$ is locally compact, $e(Q) \le \kappa$,  and $Q$ is KSB, then $Q$ has calibre $(\kappa^+,\omega)$.
\end{lemma}

\begin{proof}
Let $S$ be a subset of $Q$ of size $\kappa^+$.  We show $S$ has an infinite subset with an upper bound.  
As $e(Q) \le \kappa$,  $S$ is not closed and discrete, so there is a $q$ in $Q$ such that $q$ is in the closure of $S\setminus\{q\}$.  
Since $Q$ is locally compact there is a compact neighborhood $C$ of $q$. Then $C$ contains an infinite subset $S_0$ of $S$. Then $\cl{S_0}$ is compact, so by KSB, $S_0$ has an upper bound.  
\end{proof}

A space $X$ is said to have a $P$-ordered compact cover, where $P$ is a directed set, if it has a compact cover $\{K_p : p \in P\}$ where $p \le p'$ implies $K_p \subseteq K_{p'}$. Clearly if $P \tq \K(X)$ then $X$ has a $P$-ordered compact cover. 
Let us say that a space $X$ has \emph{relative calibre $(\mu,\omega)$} (in $\K(X)$) if every subset $S$ of $X$ of size $\mu$ has an infinite subset $S_0$ with compact closure. 
Clearly, if $\K(X)$ has calibre $(\mu,\omega)$ then $X$ has relative calibre $(\mu,\omega)$. 
But also easy to check:
\begin{lemma}
If a space $X$ has a $P$-ordered compact cover, where $P$ has calibre $(\mu,\omega)$ then $X$ is relative calibre $(\mu,\omega)$.
\end{lemma}

Define the \emph{extent} of a space $Y$, to be $e(Y)=\sup \{|E| : E $ is closed and discete in $Y\}$. Observe $Y$ has countable extent if and only if $e(Y) \le \aleph_0$.  
\begin{lemma}
If $X$ has relative calibre $(\kappa^+,\omega)$ then $e(X) \le \kappa$.
\end{lemma}
\begin{proof}
To show that $e(X) \le \kappa$ we need to show that no subset of $X$ of size $>\kappa$ is closed discrete. 
Well take, any $S$ a subset of $X$ of size $>\kappa$. Then by relative calibre $(\kappa^+,\omega)$, there is an infinite subset $S_0$ of $S$ with compact closure. Then $S_0$, and $S$, can not be closed and discrete.
\end{proof}

\begin{lemma}
For any space $X$ we have $\K(X) =_T \K(\K(X))$.
\end{lemma}
\begin{proof}
Since $\K(X)$ embeds as a closed subspace of $\K(\K(X))$, certainly $\K(X) \le_T \K(\K(X))$. 
For the converse, define $\phi : \K(X) \to \K(\K(X))$ by $\phi(K)=\down{K} = \{ L \in \K(X) : L \subseteq K\}$. Recalling that $\K(\K(X))$ is DK, we see $\phi$ does map into $\K(\K(X))$.  
Clearly $\phi$ is order-preserving.
For any $\K$ in $\K(\K(X))$, we know $\bigcup \K$ is a compact subset of $X$, then $\phi(\bigcup \K) \supseteq \K$, and  so $\phi$ has cofinal image.
\end{proof}

\begin{lemma}\label{l:ext} Let $X$ be locally compact. 
Then $\K(X)$ has calibre $(\kappa^+,\omega)$ if and only if $e(\K(X)) \le \kappa$. 
\end{lemma}
\begin{proof}
Suppose, first, that $\K(X)$ has calibre $(\kappa^+,\omega)$. Then $\K(\K(X))$ has calibre $(\kappa^+,\omega)$, and $\K(X)$ has relative calibre $(\kappa^+,\omega)$ in $\K(\K(X))$. Hence $e(\K(X))\le \kappa$. 

Now suppose, $e(\K(X))\le \kappa$. Apply Lemma~\ref{l:ce_to_om1om}. 
\end{proof}

\section{`Rejecting' Directed Sets in the Compact Case}

\subsection{Totally Bounded Uniformities}

\begin{theorem}\label{th:reject1} Let $X$ be  a set. 
 For any totally bounded uniformity, $\mathcal{U}$, on $X$ we have $\mathcal{U} =_T [\kappa]^{< \omega}$ where $\kappa=\cof(\mathcal{U})$. 
 More generally, for any uniformity $\mathcal{U}$ on $X$ and totally bounded subset $S$ we have $\mathcal{U}\restriction S =_T [\kappa]^{< \omega}$ where $\kappa=\cof(\mathcal{U}\restriction S)$.
\end{theorem}
\begin{proof}
Evidently, the `more generally' claim follows from the first part, which we now prove. 

For any directed set $P$  we have $[\cof (P)]^{<\omega} \tq P$, so certainly $[\kappa]^{<\omega} \tq \mathcal{U}$. We show $[\kappa]^{<\omega} \le_T \mathcal{U}$. To do this it suffices to show that there is a  $\kappa$-sized subcollection, $\{U_\alpha : \alpha<\kappa\}$, of $\mathcal{U}$ such that no infinite subset has an upper bound, for then $\psi : [\kappa]^{<\omega} \to \mathcal{U}$ defined by $\psi (F) = \bigcap_{\alpha \in F} U_\alpha$ is a Tukey map (carries unbounded sets to unbounded sets).

We may assume $X$ is compact in the topology induced by $\mathcal{U}$, and so $\mathcal{U}=\mathcal{N}_\Delta^{X^2}$.  Indeed we can replace $X$ with the space completion, $\widehat{X}$, which is compact as $\mathcal{U}$ is totally bounded, and $\mathcal{U}$ with its completion $\widehat{\mathcal{U}}$, which is totally bounded. 
To see this, recall, see Lemma~\ref{l:complete_Tukey}, that  $\mathcal{U} =_T \widehat{\mathcal{U}}$, so, in particular, $\kappa=\cof(\mathcal{U}) = \cof(\widehat{\mathcal{U}})$. 

Construct by recursion a family of pairs $(x_{\alpha,1},x_{\alpha,2})$ in $X^2$ and $U_\alpha$ in $\mathcal{U}$, for $\alpha < \kappa$, such that  $(\bigstar)$: $x_{\alpha,i} \ne x_{\beta,j}$ if $i\ne j$ or $\alpha \ne \beta$ and, if $x_{\beta,i} \in U_\alpha[x_{\alpha,i}]$ and $x_{\alpha,i} \in U_\beta[x_{\beta,i}]$ for $i=1,2$ then $\alpha = \beta$.  
To facilitate the construction we additionally ensure that for every $\alpha$ we have $\cl{U_\alpha[x_{\alpha,1}]} \cap \cl{U_\alpha[x_{\alpha,2}]}  = \emptyset$. 

At stage $\alpha$ of the recursive construction first observe that $\{ U_{\beta}[x_{\beta,1}] \times U_{\beta}[x_{\beta,2}] : \beta < \alpha\}$ 
is not a cover of $X^2 \setminus \Delta$. Otherwise every compact (equivalently, closed) subset of $X^2\setminus \Delta$ is contained in some finite union of these open rectangles, and so 
the collection of all complements in $X^2$ of finite unions of closures of rectangles $U_{\beta}[x_{\beta,1}] \times U_{\beta}[x_{\beta,2}]$, where $\beta<\alpha$, would be a base for $\mathcal{N}_{\Delta}^{X^2}=\mathcal{U}$ of size strictly less than $\kappa$, which is the cofinality of $\mathcal{U}$ -- contradiction.
Pick $(x_{\alpha,1},x_{\alpha,2})$ in $X^2\setminus \Delta$ but not in $\bigcup_{\beta<\alpha} (U_{\beta}[x_{\beta,1}] \times U_{\beta}[x_{\beta,2}])$. 
Pick $U_\alpha$ in $\mathcal{U}$ so that $\cl{U_\alpha[x_{\alpha,1}]} \cap \cl{U_\alpha[x_{\alpha,2}]}  = \emptyset$. 
By construction, $x_{\alpha,1} \ne x_{\alpha,2}$, and for no $\beta < \alpha$ do we have $x_{\alpha,i} \in U_\beta[x_{\beta,i}]$ for $i=1,2$. Hence $(\bigstar)$ holds.

To complete the proof it remains to show that $\{U_\alpha : \alpha < \kappa\}$ contains no infinite bounded subcollection. 
For a contradiction, suppose $A$ is an infinite subset of $\kappa$ such that $U_A=\bigcap_{\alpha \in A} U_\alpha$ is in $\mathcal{U}$. Pick $U_\infty$ in the uniformity $\mathcal{U}$ so that $U_\infty^{-1} \circ U_\infty \subseteq U_A$. 
As $\mathcal{U}$ is totally bounded, fix finite $F$ such that $U_\infty[F]=X$. 
Then $\{U_\infty[a] \times U_\infty[b] : a,b \in F\}$ is a finite cover of $X^2$, so there is an infinite subset $A'$ of $A$ and $(x_{\infty,1},x_{\infty,2})$ so that for all $\alpha$ in $A'$ we have $(x_{\alpha,1},x_{\alpha,2}) \in U_\infty[x_{\infty,1}] \times U_\infty[x_{\infty,2}]$.  

Take any two distinct $\alpha, \beta$ from $A'$. 
Then, for $i=1,2$, we have $x_{\alpha,i} \in U_\infty[x_{\infty,i}]$ so 
(a) $(x_{\alpha,i},x_{\infty,i}) \in U_\infty^{-1}$, and $x_{\beta,i} \in U_\infty[x_{\infty,i}]$ so (b) $(x_{\infty,i},x_{\beta,i}) \in U_\infty$. 
From (a) and (b), for $i=1,2$, we have $(x_{\alpha,i},x_{\beta,i}) \in U_{\infty}^{-1} \circ U_\infty \subseteq U_\alpha$, so $x_{\beta,i}$ is in $U_{\alpha}[x_{\alpha,i}]$. 
Interchanging $\alpha$ and $\beta$ in the argument above, for $i=1,2$, along with $x_{\beta,i} \in U_{\alpha}[x_{\alpha,i}]$, we also have  
$x_{\alpha,i} \in U_{\beta}[x_{\beta,i}]$, and this contradicts  $(\bigstar)$, as desired.
\end{proof}

We now re-interpret this theorem, which is a combinatorial statement about the Tukey type of certain  uniformities (namely the totally bounded ones) topologically and in terms of calibres. 
Recalling that $P \not\tq [\mu]^{<\omega}$ if and only if $P$ has calibre $(\mu,\omega)$, and  applying  this to $\mu=\kappa^+$, the following is immediate.

\begin{theorem}\label{th:reject2}\label{cor:unif_cal}
Suppose $\mathcal{U}$ is a compatible uniformity for a space $X$. If $A$ is a totally bounded subset and $\mathcal{U}$ is calibre $(\kappa^+,\omega)$ then $w(A) \le \kappa$. 
\end{theorem}

Totally bounded subsets of uniform spaces are also called \emph{precompact} subsets (because they are precisely those subsets whose closure in the completion is compact). Pseudocompact, and,  \textsl{a fortiori}, countably compact and compact subsets are totally bounded with respect to any compatible uniformity. 
\begin{corollary}
Suppose $\mathcal{U}$ is a compatible uniformity for a space $X$. If $A$ is a pseudocompact subspace  and $\mathcal{U}$ is calibre $(\kappa^+,\omega)$ then $w(A) \le \kappa$. 

In particular, if $\mathcal{U}$ is calibre $(\omega_1,\omega)$ then every pseudocompact (and, countably compact or compact)  subspace of $X$ is (compact, second countable and) metrizable.
\end{corollary}
\begin{corollary}\label{cor:precompact_in_tg}
Let $G$ be a topological group with identity $e$. 
If $A$ is a precompact subset of $G$, and $\mathcal{N}_e^G$ is calibre $(\kappa^+,\omega)$ then $w(A) \le \kappa$.
\end{corollary}

Compact spaces have a unique compatible uniformity, which coincides with all neighborhoods of the diagonal, and so with complements of compact subsets of the square disjoint from the diagonal. Consequently we can rephrase again for compact spaces, or subspaces,  without (direct) reference to a uniformity. 
Further recall (Lemma~\ref{l:ext}), if $Y$ is locally compact then $\K(Y)$ has calibre $(\kappa^+,\omega)$ if and only if the extent of $\K(Y)$, $e(\K(Y))$, is no more than  $\kappa$.

\begin{theorem}\label{th:reject3}\label{cor:cpt_wt} Let $X$ be a space.

(1) Suppose $X$ is compact. Then $w(X)=\kappa$ if and only if  $\mathcal{N}_\Delta^{X^2} =_T [\kappa]^{<\omega}$. 

Alternatively phrased, $w(X) \le \kappa$ if and only if $\mathcal{N}_\Delta^{X^2}$ has calibre $(\kappa^+,\omega)$.
And a  second alternative phrasing, $w(X)=e(\K(X^2 \setminus \Delta))$.

(2) More generally, if $A$ is a compact subset of a space $X$, and $\mathcal{N}_\Delta^{X^2}$ is calibre $(\kappa^+,\omega)$ then $w(A) \le \kappa$. 

Both (1) and (2)  hold for any directed set Tukey equivalent to $\mathcal{N}_\Delta^{X^2}$, including $\mathcal{U}_X$, or any other compatible uniformity, and $\K(X^2 \setminus \Delta)$.
\end{theorem}

\subsection{Variations and Problems}\label{ssec:VP}

Comparing what we can prove about pseudocompact and countably compact subsets of uniform spaces with calibre $(\omega_1,\omega)$, with what we can show for compact subsets of a space, $X$, with $\mathcal{N}_\Delta^{X^2}$ or $\K(X^2 \setminus \Delta)$ having  calibre $(\omega_1,\omega)$, the following questions naturally arise. 

\begin{ques} Let $X$ be a space. 
Suppose one of the directed sets specified has calibre $(\omega_1,\omega)$: 
 (A) $\mathcal{N}_\Delta^{X^2}$ or (B) $\K(X^2\setminus \Delta)$.

(i) If $X$ is countably compact then is $X$ (compact, second countable and) metrizable?

(ii) If $X$ is  pseudocompact then is $X$ (compact, second countable and) metrizable?

(iii) If $A$ is a countably compact (closed) subset of $X$ then is $A$ (compact, second countable and) metrizable?

(iv) If $A$ is a pseudocompact (closed) subset of $X$ then is $A$ (compact, second countable and) metrizable?
\end{ques}

Next we answer (A)(i) positively. The remaining questions are wide open.

\begin{proposition}
Let $X$ be a space such that $\mathcal{N}_\Delta^{X^2}$ has calibre $(\omega_1,\omega)$. If $X$ is countably compact  then $X$ is compact (and metrizable).
\end{proposition}
\begin{proof} We show the contra-positive. So suppose $X$ is not compact. 
Then either $X$ is not countably compact, and we are done, or $X$ is not Lindel\"{o}f, so we can find a strictly increasing sequence of open sets, $\{U_\alpha : \alpha < \kappa\}$ covering $X$ with no countable subcover (so $\kappa$ is uncountable).  Fix $y_\alpha \in U_\alpha \setminus \bigcup_{\beta < \alpha} U_\beta$. 
For each $\alpha < \kappa$, 
let $V_\alpha = X \setminus \{y_\alpha \}$ and $N_\alpha = U_\alpha^2 \cup V_\alpha^2$ , which is a neighborhood of the diagonal. As $\mathcal{N}_\Delta^{X^2}$ has calibre $(\omega_1,\omega)$, we
know there is an infinite $A \subseteq \kappa$ and a neighborhood $N$ of the diagonal such that $N_\alpha \supseteq N$
for all $\alpha \in A$. Without loss of generality, we may assume $N = \bigcup_{W \in \mathcal{W}} W^2$ for some open cover $\mathcal{W}$ of $X$.
We will show that each member of $\mathcal{W}$ contains at most one point of the infinite set
$\{y_\alpha : \alpha \in A\}$. 
Indeed, if $y_\alpha , y_\beta \in W$ for some $W$ from $\mathcal{W}$ and distinct $\alpha,\beta \in A$, then 
we see that $(y_\alpha , y_\beta) \in W^2 \subseteq 
N \subseteq N_\alpha = U_\alpha^2 \cup V_\alpha^2$ . Now since $y_\alpha \notin V_\alpha$ , then we must have $y_\beta \in U_\alpha$ , which gives $\beta \le \alpha$.
But symmetrically,  $(y_\alpha , y_\beta ) \in U_\beta^2 \cup V_\beta^2$ , which implies that $\alpha \le \alpha$. Thus $\alpha=\beta$. 
Hence, $\mathcal{W}$
witnesses that $\{y_\alpha :\alpha \in A\}$ is an infinite closed discrete subset of $X$, and thus $X$ is not countably compact. 
\end{proof}

\subsection{$P$-Bases}

Recall from the Introduction that a space $X$ is said to have a $P$-base, where $P$ is a directed set, if  every point $x$ has a neighborhood base, $\{U_p : p \in P\}$, where $U_p \subseteq U_{p'}$ if $p \ge p'$. 
Also recall that a topological space $X$ is scattered if each non-empty subspace of $X$ has an isolated point. It is known that any compact scattered space is zero-dimensional.
Scattered spaces can be stratified by the scattered height, as follows. 
For any subspace $A$ of a space $X$, let $A'$ be the set of all non-isolated
points of $A$. It is straightforward to see that $A'$ is a closed subset of $A$. Let $X^{(0)} = X$ and define $X^{(\alpha)} = \bigcap_{\beta<\alpha} (X^{(\beta)} ) '$ for each $\alpha > 0$. Then a space $X$ is  scattered if $X ^{(\alpha)} = \emptyset$ for some ordinal $\alpha$. If $X$ is scattered then for each of its points, $x$, there exists a unique ordinal $h(x)$ such that $x\in X^{(h(x))}\setminus X^{(h(x)+1)}$. The ordinal $h(X)=\sup\{h(x): x\in X\}$  is called the scattered
height of $X$ and is denoted by $h(X)$.  Also, it  straightforward to show that for any compact scattered space $X$, $X^{(h(X))}$ is a non-empty finite subset.

 The ordinal space $\omega_1+1$ is compact, scattered, of scattered height $\omega_1$, with a  $P$-base for $P=\K(\mathbb{Q})$, (and consistently, $P=\omega^\omega$), and so for $P$ with calibre $(\omega_1,\omega)$. 
 But it is an interesting open question, raised in \cite[Question~ 3.3]{Feng_Pbase} , whether every compact scattered space, with countable scattered height and a $P$-base where $P$ is calibre $(\omega_1,\omega)$, is countable. 
 We show the answer is positive if the space is additionally hereditarily meta-Lindel\"{o}f (every open cover of any subspace has a point-countable open refinement).

\begin{theorem}\label{th:hM}  Let $X$ be a hereditarily meta-Lindel\"{o}f, compact, scattered space  with countable scattered height. 
If $X$ has a $P$-base with $P$ having calibre $(\omega_1, \omega)$, then $X$ is countable, hence metrizable.
\end{theorem}

This follows immediately from the next technical result.
\begin{theorem} Let $X$ be a compact scattered space with countable scattered height. Suppose that, to each $x$ in $X$, we can assign a clopen neighborhood $U_x$ such that $U_x\cap X^{(h(x))}=\{x\}$ and $\{U_x: x\in X\}$ is point-countable. 
If $X$ has a $P$-base with $P$ having calibre $(\omega_1, \omega)$, then $X$ is countable, hence metrizable. \end{theorem}

\begin{proof} The proof is by induction on the scattered height. In \cite{DowFeng} it is shown that the result holds when the scattered height is finite. 
Assume, then, that the scattered height of $X$ is $\alpha$ with $\alpha<\omega_1$ and the result holds for any compact space with scattered height $<\alpha$. We show that $X$ is countable. 

Suppose, first, $\alpha$ is a successor , $\alpha=\alpha^-+1$. Let $Y=X^{(\geq \alpha^-)}$. Clearly $Y$ is a compact subspace of $X$ with scattered height $2$, hence it is countable. Then it is straightforward to verify that $X$ is countable. 

Now suppose that $\alpha$ is a limit. Without loss of generality, we assume that $X^{(\alpha)}$ is a singleton, denoted by $x_\ast$. As $X$ has a $P$-base we can fix a clopen base $\{B_p: p\in P\}$ at $x_\ast$, where $B_p \subseteq B_{p'}$ if $p \ge p'$. For each $p\in P$, $K_p=X\setminus B_p $ is compact and with scattered height $\beta_p$ which is clearly $<\alpha$, hence it is countable and also $K_p^{(\beta_p)}$ is finite. Note that $\{K_p: p\in P\}$ is a $P$-ordered compact cover of $X\setminus \{x_\ast\}$.  

Assume, for a contradiction, that $X$ is uncountable. Let $\{U_x: x\in X\}$ be the collection of clopen sets satisfying the conditions in the statement of this theorem. Clearly for each $x\in X\setminus \{x_\ast\}$, $U_x$ is compact with scattered height $<\alpha$,  hence is countable. Also, without loss of generality, we assume that $X\setminus X^{(1)}$ is uncountable. For each $y\in X\setminus X^{(1)}$, let $C_y=\{x\in X\setminus \{x_\ast\}: y\in U_x\}$ which is clearly countable and $D_y=\bigcup\{U_x: x\in C_y\}$ which is countable too since each $U_x$ for $x\in C_y$ is countable. By a transfinite induction, we pick a set $\{y_\lambda: \lambda<\omega_1\}\subseteq X\setminus X^{(1)}$ such that $y_\lambda \notin \bigcup\{D_{y_\gamma}: \gamma <\lambda\}$. Such a set exists because $X\setminus X^{(1)}$ is uncountable. Also, it is clear that the cardinality of $U_x\cap \{y_\lambda: \lambda<\omega_1\}$ is at most $1$.  For each $\lambda<\omega_1$, pick $p_\lambda\in P$ such that $y_\lambda\in K_{p_\lambda}$. Since $P$ has calibre $(\omega_1, \omega)$, there is a countable subset $\{\lambda_n:n\in \omega\}$ in $\omega_1$ such that $\{p_{\lambda_n}: n\in \omega\}$ is bounded above. We denote $p_\ast$ to be an upper bound of $\{p_{\lambda_n}: n\in \omega\}$ in $P$. Then $\{y_{\lambda_n}: n\in \omega\}\subseteq K_{p_\ast}$. Since $K_{p_\ast}$ is compact, there is a finite subset $\{x_i: i\leq m\}$ of $K_{p_\ast}$ such that $K_{p_\ast}\subseteq \bigcup \{U_{x_i}: i\leq m\}$. Then, $\{y_{\lambda_n}: n\in \omega\} \subseteq  \{U_{x_i}: i\leq m\}$ which is a contradiction because each $y_{\lambda_n}$ is at most in one of the $\{U_{x_i}: i\leq m\}$. This finishes the proof. \end{proof}

\section{Examples and Applications}

\subsection{Strong Diagonals}

We present as an application of Theorem~\ref{cor:cpt_wt} the following partial solution to \cite[Problem~4.1]{Sanchez}. Partial because Sanchez asked about compact spaces, $X$, with an `$M$-diagonal' which means that $X^2 \setminus \Delta$ has a $\K(M)$-ordered compact cover, and we require a `strong $M$-diagonal'. A space $X$ has a  \emph{strong $M$-diagonal} if $\K(M) \tq \K(X^2 \setminus \Delta)$.
\begin{theorem} Suppose $M$ is a metric space and $K$ is a compact
space with a strong $M$-diagonal. Then $w(K ) \le w(M)$.
\end{theorem}
\begin{proof}
First note that because every convergent sequence (say, $(K_n)_n$ converging to $K_\infty$) in $\K(M)$ is bounded (by $K_\infty \cup \bigcup_n K_n$) and every subset of $\K(M)$ of size strictly bigger than $w(M)$ has a proper limit point, we have that $\K(M)$ is calibre $(w(M)^+,\omega)$. 
Then, because $\K(M) \tq \K(X^2 \setminus \Delta)$ we see  $\mathcal{N}_{\Delta}^{X^2}$ is calibre $(w(M)^+,\omega)$. 
Now apply Theorem~\ref{cor:cpt_wt}.
\end{proof}

Sanchez's \cite[Problem~4.2]{Sanchez} remains an interesting open problem, however the next example gives a counter-example, at least consistently, to his remaining Problems~4.3-4.10.
\begin{exam} There are spaces $X$ and $Y$ with $\K(Y) \tq \K(X^2\setminus \Delta)$ (i.e. `$X$ has a strong $Y$-diagonal') such that:

(1) $X$ is a first countable compact space,  with $w(X)=\mathfrak{c}$, and $Y$ is $\sigma$-compact, first countable and cosmic;

(2) assuming $(\neg CH)$, taking $\kappa=\aleph_1$, $X$ is first countable and compact, $Y$ is the union of $\kappa$-many compact sets, and $nw(Y) \le  \kappa < w(X)$. 
\end{exam}
\begin{proof}
Let $Y$ be the bowtie space. Let $X$ be any first countable, compact space of weight $\mathfrak{c}$ (the double arrow space, for example, or the Alexandrov duplicate of $[0,1]$). 
Then $X$ is compact, first countable and, assuming $(\neg CH)$,  $w(X)=\mathfrak{c} > \aleph_1=\kappa$. While $Y$ is $\sigma$-compact (hence is the union of $\aleph_1$-many compact subsets), $nw(Y)=\aleph_0 \le \aleph_1$.
Since $X$ has weight $\mathfrak{c}$, by  Theorem~\ref{cor:cpt_wt} we know $\K(X^2\setminus \Delta)=_T [\mathfrak{c}]^{<\omega}$.
But from \cite{DowFeng} we know, as $Y$ is the bowtie space, $\K(Y) =_T [\mathfrak{c}]^{<\omega}$ Hence, $\K(Y) \tq \K(X^2\setminus \Delta)$, as claimed.
\end{proof}

\subsection{General Topological Groups}

Solving, negatively, Question~6.5 from \cite{Gab_etal} and Question~5.9 from \cite{Feng_Pbase} we offer the following example.
\begin{exam}
There is an Abelian topological group $G$, with identity $0$, such that all precompact subsets of $G$ are metrizable and $\chi(G)=\cof(\mathcal{N}_0^G)=\mathfrak{d}$ but for no separable metrizable space $M$ do we have $\K(M) \tq \mathcal{N}_0^G$, in particular, 
$\omega^\omega \not\tq \mathcal{N}_0^G$ ($G$ does not have an $\omega^\omega$-base).
\end{exam}
\begin{proof}
Let $P= \sum \omega^{\omega_1}$ (the $\Sigma$ product of $\omega_1$-many copies of $\omega$). Then $P$ is calibre $(\omega_1,\omega)$ \cite{GartMama}. And the cofinality of $P$ is $\mathfrak{d}$ (for each $\alpha < \omega_1$ take a cofinal family, $C_\alpha$, of size $\le\mathfrak{d}$ of $\alpha^\omega$, extend each element of $C_\alpha$ to have value $0$ for all $\beta \ge \alpha$, giving a $C_\alpha'$ contained in $\sum \omega^{\omega_1}$, and union together the $C_\alpha'$'s). 
For no separable metrizable $M$ \cite{GartMama} do we have $\K(M) \tq \sum \omega^{\omega_1}$, in particular not $\omega^\omega$. 

Now construct the topological group, $G=G_P$, as in Theorem~\ref{represent}(3). Corollary~\ref{cor:precompact_in_tg} guarantees that precompact subsets of $G$ are metrizable. 
\end{proof}

%\section*{Declarations}

%\subsection*{Competing Interests} The authors have no competing interests (funding, employment, financial or non-financial).


\begin{thebibliography}{AA}


%\bibitem{ArOkPe} A.V. Arhangel’skii, O.G. Okunev,  V.G. Pestov, \emph{Free topological groups over metrizable spaces}, Topology Appl. 33 (1989), 63--76.

\bibitem{Banakh2019} T. Banakh,
   \emph{Topological spaces with an $\omega^{\omega}$-base},
   Dissertationes Math. 538 (2019), 141 pp. 

\bibitem{BL18} T. Banakh, A. Leiderman, 
   \emph{$\omega^\omega$-dominated function spaces and $\omega^\omega$-bases in free objects of topological algebra}, Topology and Appl. 241 (2018), 203--241. 
   %issn={0012-3862},
   %review={\MR{3942223}},
   %doi={10.4064/dm762-4-2018},



\bibitem{DowFeng} A. Dow, Z. Feng, \emph{Compact spaces with a $P$-base}, 
Indag. Math. (N.S.) 32 (2021), no. 4, 777-791. 


\bibitem{engelking} R. Engelking, \emph{General topology}, Revised and completed edition. Sigma Series in Pure Mathematics, vol. 6. Heldermann Verlag, Berlin, 1989.

\bibitem{EFHST} A. Eshed, V. Ferrer, S. Hernandez, P. Szewczak, B. Tsaban, \emph{A Classification of the Cofinal Structure of Precompacta}, Annals of Pure and Applied Logic,
Volume 171, Issue 8, August–September 2020, 102810.



\bibitem{Feng_Pbase} Z. Feng,
\emph{$P$-bases and topological groups}, 
Proc. Amer. Math. Soc. 150 (2022), no. 2, 877-889. 


\bibitem{Fr2}  D.H. Fremlin, \emph{Families of compact sets and Tukey's ordering}, Atti Sem. Mat. Fis. Univ. Modena 39  no. 1 (1991) 29-50. 

\bibitem{Fr3}  D.H. Fremlin, \emph{The partially ordered sets of measure theory and Tukey's ordering}, Note Mat. 11 (1991)  177-214. Dedicated to the memory of
Professor Gottfried K\"{o}the.

%\bibitem{Frem}  D.H. Fremlin, \emph{Measure Theory}, Volume 5. (Torres Fremlin 2000)



\bibitem{Gab_etal} S. Gabriyelyan, J. Kakol, A. Leiderman,
\emph{On topological groups with a small base and metrizability}, 
Fund. Math. 229 (2015), no. 2, 129–158. 


%\bibitem{Gart_diversity} P. Gartside, \emph{Tukey order and diversity of free Abelian topological groups}, J. Pure Appl. Algebra 225 (2021), no. 10, Paper No. 106712, 13 pp. 




\bibitem{GM} P.M. Gartside, A.  Mamatelashvili, 
\emph{Tukey order on compact subsets of separable metric spaces}, Journal of Symbolic Logic 81 no. 1, (2016) 181-200. 

\bibitem{GartMama} P. Gartside, A. Mamatelashvili, \emph{The Tukey order and subsets of $\omega_1$},
Order 35 (2018), no. 1, 139-155.

\bibitem{GM2} P.M. Gartside, A.  Mamatelashvili, 
\emph{Tukey Order, Calibres and the Rationals}, Annals of Pure and Applied Logic,
Volume 172, Issue 1, January 2021, 102873.


\bibitem{GartMorgan_PN} P. Gartside, J. Morgan, \emph{Local networks for function spaces}, Houston J. Math. 45 (2019), no. 3, 893-923. 




%\bibitem{NickTka} P. Nickolas,  M. Tkachenko, \emph{The character of free topological groups I}, Applied General Topology, vol. 6, no. 1 (2005), 15-41.


\bibitem{Sanchez} D.G. S\'{a}nchez, \emph{Spaces with an $M$-diagonal}, 
Rev. R. Acad. Cienc. Exactas F\'{i}s. Nat. Ser. A Mat. RACSAM 114 (2020), no. 1, Paper No. 16, 9 pp. 




%\bibitem{Sikorski} R. Sikorski, \emph{Remarks on some topological spaces of high power}, Fund. Math. 37 (1950) 125-136.


\bibitem{Tukey} J. Tukey, \emph{Convergence and unifomity in topology}, Ann. Math Studies, 2 (Princeton University Press, Princeton  1940)


\end{thebibliography}
\end{document}